\documentclass[12pt]{article}

\usepackage{amsmath}
\usepackage{amsfonts}
\usepackage{amssymb}
\usepackage{amsthm}
\usepackage{hyperref}
\usepackage{enumerate}
\usepackage{cleveref}
\usepackage{mathrsfs}
\usepackage{geometry}

\newtheorem{theorem}{Theorem}[section]
\newtheorem{lemma}[theorem]{Lemma}
\newtheorem{proposition}[theorem]{Proposition}

\newtheorem{definition}[theorem]{Definition}
\newtheorem{remark}[theorem]{Remark}

\Crefname{conjecture}{Conjecture}{Conjectures}

\theoremstyle{plain}

\theoremstyle{plain}

\newcommand{\N}{\mathbb{N}}

\newcommand{\Z}{\mathbb{Z}}

\newcommand{\R}{\mathbb{R}}
\newcommand{\C}{\mathbb{C}}

\newcommand{\bbH}{\mathbb{H}}

\renewcommand{\Re}{\operatorname{Re}}

\newcommand{\sech}{\operatorname{sech}}

\newcommand{\calG}{\mathcal{G}}

\numberwithin{equation}{section}

\author{Jehanne Dousse and Michael H. Mertens}
\title{Asymptotic Formulae for Partition Ranks}

\begin{document}

\bibliographystyle{amsplain}
\maketitle
\begin{abstract}
Using an extension of Wright's version of the circle method, we obtain asymptotic formulae for partition ranks similar to formulae for partition cranks which where conjectured by F. Dyson and recently proved by the first author and K. Bringmann.
\end{abstract}
\section{Introduction and statement of results}
A partition of $n$ is a non-increasing sequence of natural numbers whose sum is $n$.
For example, there are $5$ partitions of $4$: $4$, $3+1$, $2+2$, $2+1+1$ and $1+1+1+1$. Let $p(n)$ denote the number of partitions of $n$.
One of the most beautiful theorems in partition theory is Ramanujan's congruences for $p(n)$. He proved~\cite{Rama} that for all $n \geq 0$,
\begin{align*}
p(5n+4) \equiv& 0 \pmod{5},\\
p(7n+5) \equiv& 0 \pmod{7},\\
p(11n+6) \equiv& 0 \pmod{11}.
\end{align*}

Dyson~\cite{Dyson44} introduced the rank, defined as the largest part of a partition minus the number of its parts, in order to explain the congruences modulo $5$ and $7$ combinatorially. He conjectured that for all $n$, the partitions of $5n+4$ (resp. $7n+5$) can be divided in $5$ (resp. $7$) different classes of same size according to their rank modulo $5$ (resp. $7$). This was later proved by Atkin and Swinnerton-Dyer~\cite{ASD54}.

However the rank fails to explain the congruences modulo $11$. Therefore Dyson conjectured the existence of another statistic which he called the ``crank'' which would give a combinatorial explanation for all the Ramanujan congruences. The crank was later found by Andrews and Garvan \cite{AG, Ga}. If for a partition $\lambda$, $o(\lambda)$ denotes the number of ones in $\lambda$, and $\mu(\lambda)$ is the number of parts strictly larger than $o(\lambda)$, then the \emph{crank} of $\lambda$ is defined as
$$
\text{crank} (\lambda) := \left\{ \begin{array}{cc}  \text{largest part of $\lambda$}& \text{if $o(\lambda) =0$,} \\ \mu (\lambda) - o(\lambda) & \text{if $o(\lambda)>0$.} \end{array}\right.
$$
Denote by $M(m,n)$ the number of partitions of $n$ with crank $m$, and by $N(m,n)$ the number of partitions of $n$ with rank $m$.

The first author and Bringmann~\cite{BD13} recently proved a longstanding conjecture of Dyson by using the modularity of the crank generating function and an extension to two variables of Wright's version of the circle method~\cite{Wright}.

\begin{theorem}[Bringmann-Dousse]
If $|m|\leq \frac1{\pi\sqrt{6}}\sqrt{n}\log n$, we have as $n\to\infty$
\begin{equation}\label{crankas}
M(m, n)=\frac{\beta}{4}\mathrm{sech}^2\left(\frac{\beta m}{2}\right) p(n)\left(1+O\left(\beta^{\frac12}|m|^{\frac13}\right)\right),
\end{equation}
where $\beta:=\tfrac{\pi}{\sqrt{6n}}$.
\end{theorem}

For the rank the situation is more complicated since the generating function is not modular but mock modular, which means roughly that there exists some non-holomorphic function such that its sum with the generating function has nice modular properties. Nonetheless it is possible to apply a method similar to \cite{BD13} in this case. This way we prove that the same formula also holds for the rank.
\begin{theorem}\label{theo:main}
If $|m| \leq \frac{\sqrt{n}\log n}{\pi\sqrt{6}}$, we have as $n\rightarrow \infty$
\[N(m,n)=\frac{\beta}{4}\mathrm{sech}^2\left(\frac{\beta m}{2}\right) p(n)\left(1+O\left(\beta^{\frac12}|m|^{\frac13}\right)\right).\]
\end{theorem}
\begin{remark}
As in~\cite{BD13}, we could in fact replace the error term by $O( \beta^{\frac{1}{2}} m \alpha^2(m))$ for any $\alpha(m)$ such that $\frac{\log n}{n^{\frac{1}{4}}} = o \left(\alpha(m)\right)$ for all $|m|\leq \frac1{\pi\sqrt{6}}\sqrt{n}\log n$ and $\beta m \alpha(m) \rightarrow 0$ as $n \rightarrow \infty$. Here we chose $\alpha(m) = |m|^{-\frac{1}{3}}$ to avoid complicated expressions in the proof.
\end{remark}
\begin{remark}
After~\cite{BD13}, and simultaneously and independently to this paper, Parry and Rhoades~\cite{PR14} proved that the same formula holds for all of Garvan's $k$-ranks. The crank corresponds to the case $k=1$ and the rank to $k=2$. Their proof uses a completely different method: they use a sieving technique and do not rely on the modularity of the generating function.
\end{remark}
The rest of this paper is organized as follows: in \Cref{sec:prelim} we recall some important facts about Appell-Lerch sums, Mordell integrals, and also Euler polynomials, which are used in \Cref{sec:Trans} to prove some preliminary estimates for the rank generating function. In \Cref{sec:Asympt}, we use these results to prove the estimates close to and far from the dominant pole which we need in \Cref{sec:Circle} to establish our main result \Cref{theo:main}.

\section*{Acknowledgements}
The second author's research is supported by the DFG Graduiertenkolleg 1269 ``Global Structures in Geometry and Analysis'' at the University of Cologne. For many helpful discussions and comments on earlier versions of this paper the authors want to thank Kathrin Bringmann, Wolf Jung, Jeremy Lovejoy, Karl Mahlburg, and Jos\'e Miguel Zapata Rol\'on. 


\section{Preliminaries}\label{sec:prelim}
\subsection{(Mock) modular forms}
A key ingredient in the proof of our main theorem is the (mock) modularity of the rank generating function, defined as follows (throughout, if not specified elsewise, we always assume $\tau\in\bbH$, $z\in\R$, $q:=e^{2\pi i\tau}$, and $\zeta:=e^{2\pi iz}$),
\begin{equation}
\label{eq:rankgen}
R(z;\tau):=\sum\limits_{n=0}^\infty\sum\limits_{m\in\Z} N(m,n)\zeta^mq^n=\frac{1-\zeta}{(q)_\infty}\sum\limits_{n\in\Z} \frac{(-1)^nq^{\frac{3n^2+n}{2}}}{1-\zeta q^n}.
\end{equation}
Let us further define
\begin{equation}
\label{eq:eta}
\eta(\tau):=q^\frac{1}{24}(q)_\infty=q^\frac{1}{24}\prod\limits_{n=1}^\infty (1-q^n),
\end{equation}
and
\begin{equation}
\label{eq:theta}
\theta(z;\tau):=iq^\frac{1}{8}\zeta^{\frac 12}\prod\limits_{n=1}^\infty (1-q^n)(1-\zeta q^{n})(1-\zeta^{-1}q^{n-1}).
\end{equation}

In this section we are going to collect transformation properties for $\eta$ and $\theta$ and recall the definition and most important properties of Appell-Lerch sums as studied by S. Zwegers in \cite{ZwegersDiss}.
\begin{lemma}
\label{lem:transfo}
For $\eta$ and $\theta$ as in \eqref{eq:eta} and \eqref{eq:theta} we have the following transformation laws,
\begin{align}\label{eq:etatrafo}
\eta\left(-\frac{1}{\tau}\right)&=\sqrt{-i\tau}\eta(\tau),\\
\theta\left(\frac z\tau;-\frac 1\tau\right)&=-i\sqrt{-i\tau}e^{\frac{\pi i z^2}{\tau}}\theta(z;\tau),
\end{align}
where $\sqrt{\bullet}$ denotes the principal branch of the holomorphic square-root.
\end{lemma}
Following Chapter 1 of \cite{ZwegersDiss} we define the following.
\begin{definition}
\begin{enumerate}[(i)]
\item For $z\in\C$ and $\tau\in\bbH$, we define the \emph{Mordell integral} as
\[h(z)=h(z;\tau)=\int\limits_{-\infty}^\infty \frac{e^{\pi i\tau w^2-2\pi zw}}{\cosh(\pi w)}dw.\]
\item For $\tau\in\bbH$ and $u,v\in\C\setminus (\Z\oplus\Z\tau)$, we call the expression
\begin{equation}
\label{eq:mu}
A_1(u,v;\tau)=e^{\pi i u}\sum\limits_{n\in\Z}\frac{(-1)^nq^\frac{n^2+n}{2}e^{2\pi inv}}{1-e^{2\pi i u}q^n}
\end{equation}
an \emph{Appell-Lerch sum}. We also call $\mu(u,v;\tau):=\tfrac{A_1(u,v;\tau)}{\theta(v;\tau)}$ a normalized Appell-Lerch sum. 
\end{enumerate}
\end{definition}
We need some transformation properties of these functions:
\begin{lemma}[cf. Proposition 1.2 in \cite{ZwegersDiss}]\label{lem:htrafo}
The Mordell integral has the following properties:
\begin{enumerate}[(i)]
\item $h(z)+e^{-2\pi iz-\pi i\tau}h(z+\tau)=2\zeta^{-\frac 12}q^{-\frac 18}$,
\item $h\left(\frac z\tau;-\frac 1\tau\right)=\sqrt{-i\tau}e^{-\frac{\pi iz^2}{\tau}}h(z;\tau)$.
\end{enumerate}
\end{lemma}
\begin{lemma}[cf. Proposition 1.4 and 1.5 in \cite{ZwegersDiss}]\label{lem:mutrafo}
\begin{enumerate}[(i)]
\item One has 
\[\mu(-u,-v)=\mu(u,v).\]
\item Under modular inversion, the Appell-Lerch sum has the following transformation law,
\[\frac{1}{\sqrt{-i\tau}}e^{\frac{\pi i(u-v)^2}{\tau}}\mu\left(\frac u\tau,\frac{v}{\tau};-\frac 1\tau\right)+\mu(u,v;\tau)=\frac{1}{2i}h(u-v;\tau),\]
or equivalently
\[-\frac 1\tau e^{\frac{\pi i(u^2-2uv)}{\tau}}A_1\left(\frac u\tau,\frac v\tau;-\frac 1\tau\right)+A_1(u,v;\tau)=\frac{1}{2i}h(u-v;\tau)\theta(v;\tau).\]
\end{enumerate}
\end{lemma}
\subsection{Euler polynomials and Euler numbers}
We now recall some facts about Euler polynomials.
We define the \emph{Euler polynomials} by the generating function
\begin{equation}\label{eq:Eulerpoly}
\frac{2e^{xz}}{e^z+1}=:\sum\limits_{k=0}^\infty E_k(x)\frac{z^k}{k!}.
\end{equation}
Let us recall two lemmas from~\cite{BD13} which will be useful in our proof.
\begin{lemma}\label{2.2}
We have
$$
-\frac12 \text{\rm sech}^2 \left( \frac{t}2\right) = \sum_{r=0}^\infty E_{2r+1} (0) \frac{t^{2r}}{(2r)!}.
$$
\end{lemma}
\begin{lemma}\label{EulerLemma}
Setting for $j\in\N_0$
\begin{equation}\label{Eulerintegral}
\mathcal{E}_j := \int_0^\infty \frac{z^{2j+1}}{\sinh (\pi z)} dz,
\end{equation}
we get
$$
\mathcal{E}_j = \frac{(-1)^{j+1} E_{2j+1} (0)}{2}.
$$
\end{lemma}


\section{Transformation Formulae}\label{sec:Trans}
In this section, we split $R(z;\tau)$ into several summands to determine its transformation behaviour under $\tau\mapsto-\tfrac 1\tau$.
\begin{lemma}
For all $\tau \in \bbH, z \in \R$, we have
\begin{equation}
\label{eq:rankJacobi}
\begin{aligned}
R(z;\tau)=\frac{q^{\frac{1}{24}}}{\eta(\tau)}\left[\frac{i\left(\zeta^{\frac 12}-\zeta^{-\frac 12}\right)\eta^3(3\tau)}{\theta(3z;3\tau)}\right.&-\zeta^{-1}\left(\zeta^{\frac 12}-\zeta^{-\frac 12}\right)A_1(3z,-\tau;3\tau)\\
&\left.-\zeta\left(\zeta^{\frac 12}-\zeta^{-\frac 12}\right)A_1(3z,\tau;3\tau)\right]
\end{aligned}
\end{equation}
with $A_1$ as in \eqref{eq:mu}.
\end{lemma}
This was first mentioned in Theorem 7.1 of \cite{Zag07}, but contained a slight typo there. To be precise, the factor $i$ in front of the first summand was missing and the sign in front of the second and third was wrong.

Now we want to determine some asymptotic expressions for the three summands in \eqref{eq:rankJacobi}. To do, so let us write $\tau=\frac{is}{2\pi}$, $s=\beta(1+ix m^{\frac{-1}{3}})$ with $x \in \R$ satisfying $|x|\leq\frac{\pi m^{\frac13}}{\beta}$.

\begin{lemma}\label{lem:asympt} 
Assume that $|z| < \frac 13$. Then for $|x| \leq 1$, we have as $n\rightarrow\infty$

\[i\frac{\eta^3(3\tau)}{\theta(3z;3\tau)}=\frac{-i\pi e^{\frac{6\pi^2 z^2}{s}}}{3s \sinh\left(\frac{2\pi^2 z}{s}\right)}\left[1+O\left(e^{-\frac{4\pi^2}{3}\Re\left(\frac 1s\right)(1-3z)}\right)+O\left(e^{-\frac{4\pi^2}{3}\Re\left(\frac 1s\right)(1+3z)}\right)\right],\]
\end{lemma}
\begin{proof}
By the transformation formulae from Lemma~\ref{lem:transfo},
{\allowdisplaybreaks
\begin{align*}
i\frac{\eta^3(3\tau)}{\theta(3z;3\tau)}&=i\frac{\left(\frac{1}{\sqrt{-3i\tau}}\right)^3\eta^3\left(-\frac{1}{3\tau}\right)}{\frac{i}{\sqrt{-3i\tau}}e^{-\pi i\frac{(3z)^2}{3\tau}}\theta\left(\frac z\tau;-\frac{1}{3\tau}\right)}\\
											  &=i\frac{\eta^3\left(-\frac{1}{3\tau}\right)}{3\tau e^{-3\pi i\frac{z^2}{\tau}}\theta\left(\frac z\tau ;-\frac{1}{3\tau}\right)}\\
											  &=\frac{2\pi\eta^3\left(\frac{2\pi i}{3s}\right) e^{\frac{6\pi^2 z^2}{s}}}{3s\theta\left(\frac{2\pi z}{is};\frac{2\pi i}{3s}\right)}\\
											  &=\frac{2\pi e^{ \frac{6\pi^2 z^2}{s}}e^{-\frac{\pi^2}{6s}}}{3is e^{\frac{2\pi^2 z}{s}}e^{-\frac{\pi^2}{6s}}}\prod\limits_{k=1}^\infty \frac{\left(1-e^{-\frac{4\pi^2 k}{3s}}\right)^2}{\left(1-e^{\frac{4\pi^2 z}{s}-\frac{4\pi^2 k}{3s}}\right)\left(1-e^{-\frac{4\pi^2 z}{s}-\frac{4\pi^2 (k-1)}{3s}}\right)}\\
											  &=\frac{2\pi e^{\frac{6\pi^2 z^2}{s}}}{3is e^{\frac{2\pi^2 z}{s}}\left(1-e^{-\frac{4\pi^2 z}{s}}\right)}\prod\limits_{k=1}^\infty \frac{\left(1-e^{-\frac{4\pi^2 k}{3s}}\right)^2}{\left(1-e^{\frac{4\pi^2 z}{s}-\frac{4\pi^2 k}{3s}}\right)\left(1-e^{-\frac{4\pi^2 z}{s}-\frac{4\pi^2 k}{3s}}\right)}\\
											   &=\frac{2\pi e^{\frac{6\pi^2 z^2}{s}}}{3is \left(e^{\frac{2\pi^2 z}{s}}-e^{-\frac{2\pi^2 z}{s}}\right)}\prod\limits_{k=1}^\infty \frac{\left(1-e^{-\frac{4\pi^2 k}{3s}}\right)^2}{\left(1-e^{\frac{4\pi^2 z}{s}-\frac{4\pi^2 k}{3s}}\right)\left(1-e^{-\frac{4\pi^2 z}{s}-\frac{4\pi^2 k}{3s}}\right)}\\
											   &=\frac{-i\pi e^{\frac{6\pi^2 z^2}{s}}}{3s \sinh\left(\frac{2\pi^2 z}{s}\right)}\left[1+O\left(e^{-\frac{4\pi^2}{3}\Re\left(\frac 1s\right)(1-3z)}\right)+O\left(e^{-\frac{4\pi^2}{3}\Re\left(\frac 1s\right)(1+3z)}\right)\right].
\end{align*}
}
\end{proof}

Before estimating the two last summands of~\eqref{eq:rankJacobi}, we need two more lemmas about $A_1$ and $h$.

\begin{lemma}
\label{lem:A_1}
Let $z \in \R$ with $|z|< \frac13$. Then for $|x| \leq 1$, we have as $n\rightarrow\infty$
$$A_1\left(\frac{2\pi z}{is},\mp\frac 13;\frac{2\pi i}{3s}\right)=\frac{-1}{2\sinh\left(\frac{2\pi^2z}{s}\right)} + O\left(e^{\frac{-2\pi^2}{3}\Re\left(\frac{1}{s}\right)(2-3z)}\right) + O\left(e^{\frac{-2\pi^2}{3}\Re\left(\frac{1}{s}\right)(2+3z)}\right).$$
\end{lemma}
\begin{proof}
In the proof of this lemma, we assume that $\zeta$ and $q$ are such that $|\zeta q^n|<1$ if $n>0$ and $|\zeta q^{n}|>1$ if $n<0$.
By applying the geometric series 
\[\frac{1}{1-x}=\begin{cases} \sum\limits_{k=0}^\infty x^k & \text{if }|x|<1, \\ -\sum\limits_{k=1}^\infty x^{-k} & \text{if }|x|>1\end{cases}\]
we find (writing $\rho=e^{\frac{2\pi i}{3}}$)
\begin{align*}
\zeta^{-\frac 12}A_1\left(z,\mp\tfrac 13;\tau\right)=\frac{1}{1-\zeta}&+\sum\limits_{n=1}^\infty (-1)^n\rho^{\mp n}q^{\frac{n^2+n}{2}}\sum\limits_{k=0}^\infty \zeta^kq^{nk} \\
&-\sum\limits_{n=1}^\infty (-1)^n\rho^{\pm n}q^{\frac{n^2-n}{2}}\sum\limits_{k=1}^\infty\zeta^{-k} q^{(-n)\cdot (-k)}.
\end{align*} 
If we see the above as a power series in $q$, we get that when $n \rightarrow \infty$,
$$\zeta^{-\frac 12}A_1\left(z,\mp\tfrac 13;\tau\right)=\frac{1}{1-\zeta} + O(q) + O\left(\zeta^{-1}q\right).$$
Thus
$$ A_1\left(z,\mp\tfrac 13;\tau\right)=\frac{-1}{\zeta^{\frac12}-\zeta^{-\frac12}} + O(\zeta^{\frac12}q) + O\left(\zeta^{-\frac12}q\right).$$

Plugging in $\zeta=e^{\frac{4\pi^2z}{s}}$ and $q=e^{-\frac{4\pi^2}{3s}}$ (which satisfy our condition that $|\zeta q^n|<1$ if $n>0$ and $|\zeta q^{n}|>1$ if $n<0$), we find:
\begin{align*}
A_1\left(\frac{2\pi z}{is},\mp\frac 13;\frac{2\pi i}{3s}\right)&=\frac{-1}{e^{\frac{2\pi^2z}{s}}-e^{\frac{-2\pi^2z}{s}}} + O\left(e^{2\pi^2z \Re\left(\frac{1}{s}\right)}e^{-\frac{4\pi^2}{3}\Re\left(\frac{1}{s}\right)}\right)\\
&\qquad\qquad\qquad\quad+ O\left(e^{-2\pi^2z \Re\left(\frac{1}{s}\right)}e^{-\frac{4\pi^2}{3}\Re\left(\frac{1}{s}\right)}\right)\\
&=\frac{-1}{2\sinh\left(\frac{2\pi^2z}{s}\right)} + O\left(e^{\frac{-2\pi^2}{3}\Re\left(\frac{1}{s}\right)(2-3z)}\right) + O\left(e^{\frac{-2\pi^2}{3}\Re\left(\frac{1}{s}\right)(2+3z)}\right).
\end{align*}
\end{proof}

We now turn to the Mordell integral.
\begin{lemma}
\label{lem:h}
For $|x| \leq 1$, we have as $n\rightarrow\infty$ that
$$\left| h\left(3z \pm \frac{is}{2\pi} ; \frac{3is}{2\pi}\right) \right| \ll e^{\frac{-\beta}{6}}.$$
\end{lemma}
\begin{proof}
We apply Lemma 3.4 of~\cite{BMR13} with $\ell = 0$, $k=2$, $h=\mp 1$, $u=0$, $z= \frac{\pi}{3s}$ and $\alpha=3z$.
This gives
$$\left| h\left(3z \pm \frac{is}{2\pi} ; \frac{3is}{2\pi}\right) \right| \ll e^{\frac{-\pi}{18}\Re\left(\frac{3s}{\pi}\right)}.$$
The result follows.
\end{proof}

With this, we can now prove the following estimate for the Appell-Lerch sums.
\begin{lemma}\label{lem:asympA1}
For $|z|\leq \frac16$ and $|x| \leq 1$, as $n\rightarrow\infty$
\begin{align*}
A_1(3z,\mp \tau; 3 \tau)=&\frac{i\pi}{3s}\frac{\zeta^{\pm 1} e^{\frac{6\pi^2z^2}{s}}}{\sinh\left(\frac{2\pi^2z}{s}\right)} +O\left(\frac{1}{|s|^{1/2}} e^{-\frac{\pi^2}{6}\Re\left(\frac1s\right)}\right).
\end{align*}
\end{lemma}

\begin{proof}
We use the transformation properties of $A_1$ to obtain 
\begin{align*}
A_1(3z,\mp\tau;3\tau)&=\frac{1}{2i}h(3z\pm \tau;3\tau)\theta(\mp\tau;3\tau)+\frac{1}{3\tau}e^{\frac{\pi i(3z^2\pm 2z\tau)}{\tau}}A_1\left(\frac z\tau,\mp\frac{1}{3};-\frac{1}{3\tau}\right)\\
               &=\frac{1}{2i}h\left(3z\pm\frac{is}{2\pi};\frac{3is}{2\pi}\right)\theta\left(\mp\frac{is}{2\pi};\frac{3is}{2\pi}\right)+\frac{2\pi}{3is}e^{\frac{2\pi^2\left(3z^2\pm\frac{2izs}{2\pi}\right)}{s}}A_1\left(\frac{2\pi z}{is},\mp\frac 13;\frac{2\pi i}{3s}\right)\\
               &=\pm\frac{1}{2}e^{\frac s6}h\left(3z\pm\frac{is}{2\pi};\frac{3is}{2\pi}\right)\eta\left(\frac {is}{2\pi}\right)-\frac{2\pi i}{3s}e^{\frac{6\pi^2z^2}{s}} \zeta^{\pm 1} A_1\left(\frac{2\pi z}{is},\mp\frac 13;\frac{2\pi i}{3s}\right),
\end{align*}
by \Cref{lem:htrafo} and \Cref{lem:mutrafo}. In the last equality we additionally used that
\[\theta(\mp \tau;3\tau)=\pm iq^{-\frac 16}\eta(\tau),\]
which is easily deduced from the definition of $\theta$ in \eqref{eq:theta}.
By Lemma~\ref{lem:h} and Lemma~\ref{lem:transfo}, we have
\begin{align*}
\left| \frac{1}{2}e^{\frac s6}h\left(3z\pm\frac{is}{2\pi};\frac{3is}{2\pi}\right)\eta\left(\frac {is}{2\pi}\right) \right| &\ll e^{\frac{\beta}{6}-\frac{\beta}{6}} \left| \eta \left(\frac{is}{2\pi}\right)\right| \ll \frac{1}{|s|^{\frac{1}{2}}} e^{\frac{-\pi^2}{6}\Re\left(\frac{1}{s}\right)}.
\end{align*}
By Lemma~\ref{lem:A_1},
\begin{align*}
-\frac{2\pi i}{3s}e^{\frac{6\pi^2z^2}{s}} \zeta^{\pm 1} A_1\left(\frac{2\pi z}{is},\mp\frac 13;\frac{2\pi i}{3s}\right) =& \frac{\pi i}{3s} \frac{e^{\frac{6\pi^2z^2}{s}} \zeta^{\pm 1}}{\sinh\left(\frac{2\pi^2z}{s}\right)} + O\left(e^{-\pi^2\Re\left(\frac{1}{s}\right) \left(\frac43 -2z -6z^2\right)} \right)\\
&+ O\left(e^{-\pi^2\Re\left(\frac{1}{s}\right) \left(\frac43 +2z -6z^2\right)} \right)
\end{align*}
For $|z| \leq \frac{1}{6}$, $\frac43 -2z -6z^2 > \frac{1}{6}$ and $\frac43 +2z -6z^2 > \frac{1}{6}$. Therefore
\begin{align*}
e^{-\pi^2\Re\left(\frac{1}{s}\right) \left(\frac43 +2z -6z^2\right)} &\ll \frac{1}{|s|^{\frac{1}{2}}} e^{\frac{-\pi^2}{6}\Re\left(\frac{1}{s}\right)},\\
e^{-\pi^2\Re\left(\frac{1}{s}\right) \left(\frac43 -2z -6z^2\right)} &\ll \frac{1}{|s|^{\frac{1}{2}}} e^{\frac{-\pi^2}{6}\Re\left(\frac{1}{s}\right)}.
\end{align*}
Thus the dominant error term is the one coming from $\pm \frac{1}{2}e^{\frac s6}h\left(3z\pm\frac{is}{2\pi};\frac{3is}{2\pi}\right)\eta\left(\frac {is}{2\pi}\right)$. The lemma follows.
\end{proof}

\section{Asymptotic behavior}\label{sec:Asympt}
Since $N(m,n)=N(-m,n)$ for all $m$ and $n$, we assume from now on that $m \geq 0$.
In this section we want to study the asymptotic behavior of the generating function of $N(m,n)$
Let us define
\[R_m(\tau):=\int_{-\frac12}^{\frac12} R(z;\tau)e^{-2\pi imz}dz.\]
Let us recall that $\tau=\frac{is}{2\pi}$ and $s=\beta\left(1+ixm^{-\frac{1}{3}}\right)$ with $x\in\R$ satisfying $|x| \leq \frac{\pi m^{\frac13}}{\beta}$.
To simplify the forthcoming calculations, we need the following lemma.
\begin{lemma}
\label{lem:intR}
It holds that
$$R_m(\tau)= 3\frac{q^{\frac{1}{24}}}{\eta(\tau)} \int_{-\frac16}^{\frac16} g_m(z;\tau) e^{-2\pi i mz} dz,$$
where
\begin{align*}
g_m(z;\tau):=\begin{cases} -A_1(3z,\tau;3\tau)e^{3\pi iz} + A_1(3z,-\tau;3\tau)e^{-3\pi iz} & \text{ for }m\equiv 0\pmod{3}, \\
                       -A_1(3z,-\tau;3\tau)e^{-\pi iz} -i \frac{\eta^3(3\tau)}{\theta(3z;3\tau)}e^{-\pi iz} &\text{ for }m\equiv 1\pmod{3}, \\
                        A_1(3z,\tau;3\tau)e^{\pi iz} +i \frac{\eta^3(3\tau)}{\theta(3z;3\tau)}e^{\pi iz} &\text{ for }m\equiv 2\pmod{3}. \\ \end{cases}
\end{align*}
\end{lemma}
\begin{proof}
By~\eqref{eq:rankJacobi}, let us write
$$R_m(\tau) = \frac{q^{\frac{1}{24}}}{\eta(\tau)} (I_1 - I_2 -I_3),$$
where
\begin{align*}
I_1&:= \int_{-\frac12}^{\frac12} \frac{i\left(\zeta^{\frac 12}-\zeta^{-\frac 12}\right)\eta^3(3\tau)}{\theta(3z;3\tau)} e^{-2\pi imz}dz,\\
I_2&:= \int_{-\frac12}^{\frac12} \zeta^{-1}\left(\zeta^{\frac 12}-\zeta^{-\frac 12}\right)A_1(3z,-\tau;3\tau) e^{-2\pi imz}dz,\\
I_3&:= \int_{-\frac12}^{\frac12} \zeta\left(\zeta^{\frac 12}-\zeta^{-\frac 12}\right)A_1(3z,\tau;3\tau)) e^{-2\pi imz}dz.\\
\end{align*}
First, using~\eqref{eq:theta} and~\eqref{eq:mu}, let us notice that
\begin{align}
\theta(3z+1;3\tau)&=-\theta(3z;3\tau), \label{eq1}\\
A_1(3z+1,\tau;3\tau)&=-A_1(3z,\tau;3\tau), \label{eq2}\\
A_1(3z+1,-\tau;3\tau)&=-A_1(3z,-\tau;3\tau). \label{eq3}
\end{align}
Thus by~\eqref{eq1},
\begin{align*}
I_1 =& \left(\int_{-\frac{1}{2}}^{-\frac{1}{6}} + \int_{-\frac{1}{6}}^{\frac{1}{6}} + \int_{\frac{1}{6}}^{\frac{1}{2}} \right) \frac{i\left(\zeta^{\frac 12}-\zeta^{-\frac 12}\right)\eta^3(3\tau)}{\theta(3z;3\tau)} e^{-2\pi imz}dz\\
=& -i \int_{-\frac{1}{6}}^{\frac{1}{6}} \left(e^{\pi i (z -\frac13)}-e^{-\pi i (z -\frac13)}\right) \frac{\eta^3(3\tau)}{\theta(3z;3\tau)} e^{-2\pi im(z-\frac13)}dz\\
&+ i \int_{-\frac{1}{6}}^{\frac{1}{6}} \left(e^{\pi iz}-e^{-\pi iz}\right)\frac{\eta^3(3\tau)}{\theta(3z;3\tau)} e^{-2\pi imz}dz\\
&- i \int_{-\frac{1}{6}}^{\frac{1}{6}} \left(e^{\pi i (z +\frac13)}-e^{-\pi i (z +\frac13)}\right)\frac{\eta^3(3\tau)}{\theta(3z;3\tau)} e^{-2\pi im(z+\frac13)}dz\\
=& \int_{-\frac{1}{6}}^{\frac{1}{6}} \left[ e^{\pi i z}\left( -e^{\frac{\pi i}{3}(2m-1)} + 1 - e^{\frac{\pi i}{3}(-2m+1)} \right) -e^{-\pi i z}\left( -e^{\frac{\pi i}{3}(2m+1)} + 1 - e^{\frac{\pi i}{3}(-2m-1)} \right) \right]\\
&\times i \frac{\eta^3(3\tau)}{\theta(3z;3\tau)} e^{-2\pi imz} dz.
\end{align*}
Therefore
\begin{equation}
\label{I1}
\begin{aligned}
I_1=& \begin{cases} 0 & \text{ for }m\equiv 0\pmod{3}, \\
                -3i \int_{-\frac{1}{6}}^{\frac{1}{6}} \frac{\eta^3(3\tau)}{\theta(3z;3\tau)} e^{-\pi i z(2m+1)}dz &\text{ for }m\equiv 1\pmod{3}, \\
                3i \int_{-\frac{1}{6}}^{\frac{1}{6}} \frac{\eta^3(3\tau)}{\theta(3z;3\tau)} e^{-\pi i z(2m-1)}dz &\text{ for }m\equiv 2\pmod{3}. \\ \end{cases}
\end{aligned}
\end{equation}
By the same method and using~\eqref{eq2} and~\eqref{eq3}, we obtain
\begin{equation}
\label{I2}
\begin{aligned}
I_2=& \begin{cases} -3 \int_{-\frac{1}{6}}^{\frac{1}{6}} A_1(3z,-\tau;3\tau) e^{-\pi i z(2m+3)}dz & \text{ for }m\equiv 0\pmod{3}, \\
                3 \int_{-\frac{1}{6}}^{\frac{1}{6}} A_1(3z,-\tau;3\tau) e^{-\pi i z(2m+1)}dz &\text{ for }m\equiv 1\pmod{3}, \\
                0 &\text{ for }m\equiv 2\pmod{3}, \\ \end{cases}
\end{aligned}
\end{equation}
and
\begin{equation}
\label{I3}
\begin{aligned}
I_3=& \begin{cases} 3 \int_{-\frac{1}{6}}^{\frac{1}{6}} A_1(3z,\tau;3\tau) e^{-\pi i z(2m-3)}dz & \text{ for }m\equiv 0\pmod{3}, \\
                0 &\text{ for }m\equiv 1\pmod{3}, \\
                -3 \int_{-\frac{1}{6}}^{\frac{1}{6}} A_1(3z,\tau;3\tau) e^{-\pi i z(2m-1)}dz &\text{ for }m\equiv 2\pmod{3}. \\ \end{cases}
\end{aligned}
\end{equation}
The result follows.
\end{proof}

\subsection{Bounds near the dominant pole}
In this section we consider the range $|x| \leq 1$. We start by determining the main term of $g_m$.

\begin{lemma}
\label{boundgLemma}
For all $m \geq 0$ and $-\frac{1}{6} \leq z \leq \frac{1}{6}$, we have for $|x|\leq1$ as $n\rightarrow\infty$
$$
g_m \left( z; \frac{is}{2\pi}\right) = \frac{2\pi \sin(\pi z) e^{\frac{6\pi^2 z^2}{s}}}{3s \sinh\left(\frac{2\pi^2 z}{s}\right)}+O\left(\frac{1}{|s|^{\frac 12}} e^{-\frac{\pi^2}{6}\Re\left(\frac1s\right)}\right).
$$
\end{lemma}
\begin{proof}
If $m\equiv 0\pmod{3}$, we have by Lemma~\ref{lem:asympA1}
\begin{align*}
g_m(z;\tau)&= -A_1(3z,\tau;3\tau)e^{3\pi iz} + A_1(3z,-\tau;3\tau)e^{-3\pi iz}\\
&= -\frac{i\pi}{3s}\frac{e^{\pi iz} e^{\frac{6\pi^2z^2}{s}}}{\sinh\left(\frac{2\pi^2z}{s}\right)} + \frac{i\pi}{3s}\frac{e^{-\pi iz} e^{\frac{6\pi^2z^2}{s}}}{\sinh\left(\frac{2\pi^2z}{s}\right)}  +O\left(\frac{1}{|s|^{\frac 12}} e^{-\frac{\pi^2}{6}\Re\left(\frac1s\right)}\right)\\
&= \frac{i\pi}{3s}\frac{e^{\frac{6\pi^2z^2}{s}}}{\sinh\left(\frac{2\pi^2z}{s}\right)} \left(-e^{\pi i z} +e^{-\pi i z}\right) +O\left(\frac{1}{|s|^{\frac 12}} e^{-\frac{\pi^2}{6}\Re\left(\frac1s\right)}\right)\\
&= \frac{2\pi \sin(\pi z) e^{\frac{6\pi^2 z^2}{s}}}{3s \sinh\left(\frac{2\pi^2 z}{s}\right)}+O\left(\frac{1}{|s|^{\frac 12}} e^{-\frac{\pi^2}{6}\Re\left(\frac1s\right)}\right).
\end{align*}

If $m\equiv 1\pmod{3}$, we have by Lemma~\ref{lem:asympt} and Lemma~\ref{lem:asympA1}
\begin{align*}
g_m(z;\tau)&= -A_1(3z,-\tau;3\tau)e^{-\pi iz} -i \frac{\eta^3(3\tau)}{\theta(3z;3\tau)}e^{-\pi iz}\\
&= -\frac{i\pi}{3s}\frac{e^{\pi iz} e^{\frac{6\pi^2z^2}{s}}}{\sinh\left(\frac{2\pi^2z}{s}\right)} +O\left(\frac{1}{|s|^{1/2}} e^{-\frac{\pi^2}{6}\Re\left(\frac1s\right)}\right)\\
&+\frac{i\pi e^{-\pi i z} e^{\frac{6\pi^2 z^2}{s}}}{3s \sinh\left(\frac{2\pi^2 z}{s}\right)}\left[1+O\left(e^{-\frac{4\pi^2}{3}\Re\left(\frac 1s\right)(1-3z)}\right)+O\left(e^{-\frac{4\pi^2}{3}\Re\left(\frac 1s\right)(1+3z)}\right)\right] \\
&= \frac{i\pi}{3s}\frac{e^{\frac{6\pi^2z^2}{s}}}{\sinh\left(\frac{2\pi^2z}{s}\right)} \left(-e^{\pi i z} +e^{-\pi i z}\right) +O\left(\frac{1}{|s|^{1/2}} e^{-\frac{\pi^2}{6}\Re\left(\frac1s\right)}\right)\\
&= \frac{2\pi \sin(\pi z) e^{\frac{6\pi^2 z^2}{s}}}{3s \sinh\left(\frac{2\pi^2 z}{s}\right)}+O\left(\frac{1}{|s|^{1/2}} e^{-\frac{\pi^2}{6}\Re\left(\frac1s\right)}\right).
\end{align*}

Finally, if $m\equiv 2\pmod{3}$, we have by Lemma~\ref{lem:asympt} and Lemma~\ref{lem:asympA1}
\begin{align*}
g_m(z;\tau)&= A_1(3z,\tau;3\tau)e^{\pi iz} +i \frac{\eta^3(3\tau)}{\theta(3z;3\tau)}e^{\pi iz}\\
&= \frac{i\pi}{3s}\frac{e^{-\pi iz} e^{\frac{6\pi^2z^2}{s}}}{\sinh\left(\frac{2\pi^2z}{s}\right)} +O\left(\frac{1}{|s|^{1/2}} e^{-\frac{\pi^2}{6}\Re\left(\frac1s\right)}\right)\\
&-\frac{i\pi e^{\pi i z} e^{\frac{6\pi^2 z^2}{s}}}{3s \sinh\left(\frac{2\pi^2 z}{s}\right)}\left[1+O\left(e^{-\frac{4\pi^2}{3}\Re\left(\frac 1s\right)(1-3z)}\right)+O\left(e^{-\frac{4\pi^2}{3}\Re\left(\frac 1s\right)(1+3z)}\right)\right] \\
&= \frac{i\pi}{3s}\frac{e^{\frac{6\pi^2z^2}{s}}}{\sinh\left(\frac{2\pi^2z}{s}\right)} \left(e^{-\pi i z} -e^{\pi i z}\right) +O\left(\frac{1}{|s|^{1/2}} e^{-\frac{\pi^2}{6}\Re\left(\frac1s\right)}\right)\\
&= \frac{2\pi \sin(\pi z) e^{\frac{6\pi^2 z^2}{s}}}{3s \sinh\left(\frac{2\pi^2 z}{s}\right)}+O\left(\frac{1}{|s|^{1/2}} e^{-\frac{\pi^2}{6}\Re\left(\frac1s\right)}\right).
\end{align*}

\end{proof}

In view of Lemma \ref{boundgLemma} it is natural to define
\begin{align*}
\mathcal{G}_{m,1} (s) &:=  \frac{2\pi}{s} \int_{-\frac16}^{\frac16} \frac{\sin(\pi z) e^{\frac{6\pi^2 z^2}{s}}}{\sinh\left(\frac{2\pi^2 z}{s}\right)} e^{-2\pi i mz} dz,\\
\mathcal{G}_{m,2} (s) &:= 3 \int_{-\frac16}^{\frac16} \left( g_m\left(z;\frac{is}{2\pi}\right) -\frac{2\pi \sin(\pi z) e^{\frac{6\pi^2 z^2}{s}}}{3s \sinh\left(\frac{2\pi^2 z}{s}\right)}\right) e^{-2\pi i mz} dz.
\end{align*}

Thus
\begin{equation}\label{Rm}
R_m \left( \tau\right) = \frac{q^{\frac{1}{24}}}{\eta(\tau)} \left( \mathcal{G}_{m,1} (s) + \mathcal{G}_{m,2} (s) \right).
\end{equation}

Let us note that we can rewrite $\mathcal{G}_{m,1} (s)$ as
$$\mathcal{G}_{m,1} (s) =  \frac{4\pi}{s} \int_{0}^{\frac16} \frac{\sin(\pi z) e^{\frac{6\pi^2 z^2}{s}}}{\sinh\left(\frac{2\pi^2 z}{s}\right)} \cos(2\pi mz) dz.$$
 
\begin{lemma}
\label{mainGm1}
Assume that $|x|\leq 1$ and $m\leq\frac{1}{6\beta}\log n$. Then we have as $n\rightarrow \infty$
\[\calG_{m,1}(s)=\frac{s}{4}\sech^2\left(\frac{\beta m}{2}\right)+O\left(\beta^2m^{\frac23}\sech^2\left(\frac{\beta m}{2}\right)\right).\]
\end{lemma}
\begin{proof}
We use the same method as in~\cite{BD13}. Inserting the Taylor expansion of $\sin(\pi z)$, $\exp\left(\tfrac{6\pi^2z^2}{s}\right)$, and $\cos(2\pi mz)$ in the definition of $\calG_{m,1}(s)$, we find that
$$
\sin (\pi z) e^{\frac{6\pi^2 z^2}{s}} \cos (2\pi mz) = \sum_{j,\nu,r\geq 0} \frac{(-1)^{j+\nu}}{(2j+1)! (2\nu)! r!} \pi^{2j+1} (2 \pi m )^{2\nu} \left( \frac{6\pi^2}{s}\right)^r z^{2j+2\nu+2r+1}.
$$
This yields that
$$
\mathcal{G}_{m,1} (s) = \frac{4\pi}{s} \sum_{j,\nu,r\geq 0} \frac{(-1)^{j+\nu}}{(2j+1)! (2\nu)! r!} \pi^{2j+1} (2 \pi m )^{2\nu} \left( \frac{6\pi^2}{s}\right)^r \mathcal{I}_{j+\nu+r},
$$
where for $\ell\in\N_0$ we define
$$
\mathcal{I}_\ell := \int_0^{\frac16} \frac{z^{2\ell+1}}{\sinh \left( \frac{2\pi^2z}{s} \right)} dz.
$$
We next relate $\mathcal{I}_\ell$ to $\mathcal{E}_\ell$ defined in (\ref{Eulerintegral}). For this, we note that
\begin{equation}\label{splitI}
\mathcal{I}_\ell = \int_0^{\infty} \frac{z^{2\ell+1}}{\sinh \left( \frac{2\pi^2z}{s} \right)} dz - \mathcal{I}_\ell'
\end{equation}
with 
\begin{multline*}
\mathcal{I}_\ell' := \int_{\frac16}^\infty \frac{z^{2\ell+1}}{\sinh \left( \frac{2\pi^2z}{s} \right)} dz \ll \int_{\frac16}^\infty z^{2\ell+1} e^{-2\pi^2 z \Re\left(\frac1s\right)} dz \\
 \ll \left(\Re\left(\frac{1}{s}\right)\right)^{-2\ell-2} \Gamma \left( 2\ell+2; \frac{\pi^2}{3} \text{Re}\left(\frac{1}{s}\right) \right),
\end{multline*}
where $\Gamma (\alpha ; x) := \int_x^\infty e^{-w}w^{\alpha-1} dw$. Using that as $x\rightarrow \infty$
\begin{equation}\label{incompletebound}
\Gamma \left( \ell; x\right) \sim x^{\ell -1} e^{-x}
\end{equation}
thus yields that
$$
\mathcal{I}_\ell' \ll\left(\Re\left(\frac1{s}\right)\right)^{-1} e^{-\frac{\pi^2}{3} \Re\left(\frac{1}{s}\right)}\leq e^{-\frac{\pi^2}{3} \Re\left(\frac1{s}\right)}.
$$
By a substitution in \Cref{EulerLemma}, we know that
$$
\int_0^\infty \frac{z^{2\ell+1}}{\sinh \left( \frac{2\pi^2z}{s}\right)} dz =\left( \frac{s}{2\pi} \right)^{2\ell+2} \frac{(-1)^{\ell+1} E_{2\ell+1} (0)}{2}.
$$
Thus
\begin{multline*}
\mathcal{G}_{m,1} (s) =  \sum_{j,\nu,r\geq 0} \frac{(-1)^{r+1}3^r}{2^{2j+r+1} (2j+1)! (2\nu)! r!} m^{2\nu} s^{2j+2\nu+r+1} \\\times \left( E_{2j+2\nu+2r+1} (0) + O \left( |z|^{-2j-2\nu -2r-2} e^{-\frac{\pi^2}{3} \Re\left(\frac{1}{s}\right)}\right)\right) \\
= \sum_{\nu= 0}^\infty \frac{(ms)^{2\nu}}{(2\nu)!} \left( - \frac{s}{2} E_{2\nu+1} (0) + O \left( |s|^2 \right) \right) = \frac{s}{4} \text{sech}^2 \left( \frac{ms}{2} \right) + O \left( |s|^2 \cosh (ms)\right),
\end{multline*}
where for the last equality we used Lemma \ref{2.2}.
The end of the proof is now exactly the same as in Lemma 3.2 of~\cite{BD13}.
\end{proof}

We now want to bound $\mathcal{G}_{m,2} (s)$.

\begin{lemma}
\label{mainGm2}
Assume that $|x|\leq 1$. Then we have as $n \rightarrow \infty$
$$\left| \mathcal{G}_{m,2} (s)\right| \ll \frac{1}{\beta^{\frac{1}{2}}} e^{-\frac{\pi^2}{12 \beta}}.$$
\end{lemma}
\begin{proof}
By Lemma~\ref{boundgLemma}, we have
\begin{align*}
\left| \mathcal{G}_{m,2} (s)\right| &\ll \int_{-\frac{1}{6}}^{\frac{1}{6}} \left| \frac{1}{|s|^{1/2}} e^{-\frac{\pi^2}{6}\Re\left(\frac1s\right)} e^{-2\pi i mz} \right| dz\ll \frac{1}{|s|^{1/2}} e^{-\frac{\pi^2}{6}\Re\left(\frac1s\right)}.
\end{align*}
By the definition of $s$, we know that $\frac{1}{|s|^{1/2}} \leq \frac{1}{\beta^{1/2}}$. Furthermore, as $|x| \leq 1$, we have $\Re\left(\frac1s\right) \geq \frac{1}{2\beta}$. Thus
$$\left| \mathcal{G}_{m,2} (s)\right| \ll \frac{1}{\beta^{1/2}} e^{-\frac{\pi^2}{12\beta}}.$$
\end{proof}

Combining Lemma~\ref{mainGm1} and Lemma~\ref{mainGm2}, we obtain the following asymptotic estimation of $R_m(\tau)$ near the dominant pole.

\begin{proposition}
\label{close}
Assume that $|x| \leq 1$. Then we have as $n \rightarrow \infty$
$$R_m(\tau) = \frac{s^{\frac32}}{4(2\pi)^\frac 12}\operatorname{sech}^2\left(\frac{\beta m}{2}\right)e^{\frac{k\pi ^2}{6s}}+O\left(\beta^{\frac{5}{2}}m^{\frac23}\operatorname{sech}^2\left(\frac{\beta m}{2}\right)e^{\pi\sqrt{\frac{n}{6}}}\right).$$
\end{proposition}
\begin{proof}
Recall from~\eqref{Rm} that
$$
R_m(\tau)=\frac{q^\frac {1}{24}}{\eta(\tau)}\left(\mathcal{G}_{m,1}(s)+\mathcal{G}_{m,2}(s)\right).
$$
By Lemma~\ref{lem:transfo} we see that
$$
\frac{q^\frac{1}{24}}{\eta(\tau)}=\left(\frac{s}{2\pi}\right)^\frac 12 e^{\frac{\pi^2}{6s}}\left(1+O(\beta)\right).
$$
We approximate $\mathcal{G}_{m, 1}$ and $\mathcal{G}_{m, 2}$ using Lemma~\ref{mainGm1} and Lemma~\ref{mainGm2}.
The main error term comes from $\mathcal{G}_{m, 1}$. We obtain
$$R_m(\tau)= \frac{s^{\frac32}}{4(2\pi)^{\frac12}} e^{\frac{\pi^2}{6s}} \sech^2\left(\frac{\beta m}{2}\right)+O\left(s^{\frac12}\beta^2m^{\frac23}\sech^2\left(\frac{\beta m}{2}\right) e^{\frac{\pi^2}{6s}}\right).$$
The claim follows now using that
\begin{align*}
|s|&\ll \beta,\\
\Re\left(\frac1{s}\right)&\leq \frac1{\beta}=\frac{\sqrt{6n}}{\pi}.
\end{align*}
\end{proof}

\subsection{Estimates far from the dominant pole}
In the previous section, we have established bounds for the behaviour of $R_m(\tau)$ close to the pole $\tau=0$. For Wright's version of the circle method, we also need estimates far away from this pole. In this section, we consider the range $1 \leq |x| \leq \frac{\pi m^{\frac{1}{3}}}{\beta}$. First we need a lemma, which follows from an argument similar to the one in~\cite{Wright}, see also \cite[Lemma 3.5]{BD13}.
\begin{lemma}\label{lem:far}
Let $P(q)=\tfrac{q^\frac{1}{24}}{\eta(\tau)}$ be the generating function for partitions. Assume that $\tau=u+iv\in \mathbb{H}$. For $Mv \leq |u| \leq \frac{1}{2}$ and $v\rightarrow 0$, we have that
\[ \left|P(q)\right| \ll \sqrt{v} \exp \left[\frac{1}{v}\left(\frac{\pi}{12}-\frac{1}{2\pi}\left(1-\frac{1}{\sqrt{1+M^2}}\right)\right)\right].\]
\end{lemma}
\begin{proof}
Let us write the following Taylor rearrangement
\begin{align*}
\log(P(q))&=-\sum\limits_{n=1}^\infty \log(1-q^n)=\sum\limits_{n=1}^\infty\sum\limits_{m=1}^\infty \frac{q^{nm}}{m}                 =\sum\limits_{m=1}^\infty \frac{q^m}{m(1-q^m)}.
\end{align*}
Therefore we have the estimate
\begin{align*}
|\log(P(q))|&\leq\sum\limits_{m=1}^\infty \frac{|q|^m}{m|1-q^m|}\\
                       &\leq \frac{|q|}{|1-q|}-\frac{|q|}{1-|q|}+\sum\limits_{m=1}^\infty \frac{|q|^m}{m(1-|q|^m)}\\
                       &=\log(P(|q|))-|q|\left(\frac{1}{1-|q|}-\frac{1}{|1-q|}\right).
\end{align*}
For $Mv \leq |u| \leq \frac{1}{4},$ we have $\cos(2\pi u) \leq \cos(2 \pi Mv)$. Therefore
\[|1-q|^2=1-2e^{-2\pi v}\cos(2\pi u)+e^{-4\pi v}\geq 1-2e^{-2\pi v}\cos(2\pi Mv)+e^{-4\pi v}.\]
By a Taylor expansion around $v=0$ we find that
\begin{equation}
\label{eq:closefar}
|1-q|\geq 2\pi v \sqrt{1+M^2}+O(v^2).
\end{equation}
When $\frac{1}{4}\leq |u| \leq \frac{1}{2}$, we have $\cos(2\pi u) \leq 0$. Therefore
$$ |1-q| \geq 1.$$
When $v \rightarrow 0$, this is asymptotically larger than~\eqref{eq:closefar}. Hence, for all $Mv \leq |u| \leq \frac{1}{2}$,
\begin{equation}
\label{eq:|1-q|}
|1-q|\geq 2\pi v \sqrt{1+M^2}+O(v^2).
\end{equation}
Furthermore we have
\begin{equation}
\label{eq:1-q}
1 - |q| = 1-e^{-2\pi v} = 2\pi v +O(v^2).
\end{equation}
By Lemma~\ref{lem:transfo}, we have:
\begin{align*}
P(|q|)=\frac{e^{\frac{-2\pi v}{24}}}{\eta(iv)} = \sqrt{v} e^{\frac{\pi}{12v}}\left(1+O(v)\right).
\end{align*}
Thus
\begin{equation}
\label{eq:logP}
\log(P(|q|)) = \frac{\pi}{12v}+\frac{1}{2}\log(v) + O(v).
\end{equation}
Combining~\eqref{eq:|1-q|},~\eqref{eq:1-q} and~\eqref{eq:logP}, we finally obtain
\begin{align*}
|\log(P(q))|&\leq \log(P(|q|))-\frac{1}{2\pi v}\left(1-\frac{1}{\sqrt{1+M^2}}\right)(1+O(v))\\
                       &=\frac{\pi}{12v}+\frac{1}{2}\log(v) + O(v)-\frac{1}{2\pi v}\left(1-\frac{1}{\sqrt{1+M^2}}\right)+O(1)\\
                       &=\frac{1}{v}\left(\frac{\pi}{12}-\frac{1}{2\pi}\left(1-\frac{1}{\sqrt{1+M^2}}\right)\right)+\frac{1}{2}\log(v)+O(1).
\end{align*}
Exponentiating yields the desired result.
\end{proof}

We are now able to bound $|R_m(\tau)|$ away from $q=1$.

\begin{proposition}\label{far}
Assume that $1\leq|x|\leq \frac{\pi m^\frac 13}{\beta}$. Then we have as $n\rightarrow \infty$
$$
|R_{m}(\tau)| \ll \sqrt{n} \exp \left(\pi \sqrt{\frac{n}{6}} - \frac{\sqrt{6n}}{8\pi} m^{-\frac{2}{3}}\right).
$$
\end{proposition}
\begin{proof}
By~\eqref{eq:rankgen}, we have
\begin{align*}
R_m(\tau) &= P(q) \int_{-\frac{1}{2}}^{\frac{1}{2}} \left( \left(1-\zeta\right) \sum_{k\in \Z} \frac{(-1)^kq^{\frac{3k^2+k}{2}}}{1-\zeta q^k} \right) e^{-2\pi i m z} dz\\
&= P(q) \int_{-\frac{1}{2}}^{\frac{1}{2}} \left( 1+ \left(1-\zeta\right) \sum_{k \geq 1} \frac{(-1)^kq^{\frac{3k^2+k}{2}}}{1-\zeta q^k} + \left(1-\zeta^{-1}\right) \sum_{k \geq 1} \frac{(-1)^kq^{\frac{3k^2+k}{2}}}{1-\zeta^{-1} q^k}\right) e^{-2\pi i m z} dz.
\end{align*}
So we may bound $|R_m(\tau)|$ when $n \rightarrow \infty$ in the following way
\begin{align*}
|R_m(\tau)| &\ll |P(q)| \int_{-\frac{1}{2}}^{\frac{1}{2}} \sum_{k \geq 1} \frac{|q|^{\frac{3k^2+k}{2}}}{1-|q|^k} \left|e^{-2\pi i m z}\right| dz\\
&\ll |P(q)| \frac{1}{1-|q|} \sum_{k \geq 1} e^{-\beta \frac{3k^2}{2}}\\
&\ll |P(q)| \frac{1}{1-|q|} \int_{-\infty}^{\infty} e^{-\beta \frac{3x^2}{2}} dx\\
&\ll |P(q)| \frac{1}{\beta} \sqrt{\frac{2\pi}{3\beta}}\\
&\ll |P(q)| n^{\frac{3}{4}}.
\end{align*}

Now we use Lemma~\ref{lem:far} with $v=\frac{\beta}{2\pi}$, $u=\frac{\beta m^{-\frac{1}{3}}x}{2\pi}$ and $M = m^{-\frac{1}{3}}$. We obtain that for $1\leq|x|\leq \frac{\pi m^\frac 13}{\beta}$,
\begin{align*}
|P(q)| &\ll n^{-\frac{1}{4}} \exp \left[\frac{2\pi}{\beta}\left(\frac{\pi}{12}-\frac{1}{2\pi}\left(1-\frac{1}{\sqrt{1+m^{-\frac{2}{3}}}}\right)\right)\right].
\end{align*}
Therefore
\begin{align*}
|R_m(\tau)| &\ll n^{\frac{1}{2}} \exp \left[\frac{2\pi}{\beta}\left(\frac{\pi}{12}-\frac{1}{2\pi}\left(1-\frac{1}{\sqrt{1+m^{-\frac{2}{3}}}}\right)\right)\right]\\
&\ll  n^{\frac{1}{2}} \exp \left[\pi \sqrt{\frac{n}{6}} - \frac{\sqrt{6n}}{\pi} \left(1-\frac{1}{\sqrt{1+m^{-\frac{2}{3}}}}\right)\right]\\
&\ll  n^{\frac{1}{2}} \exp \left(\pi \sqrt{\frac{n}{6}} - \frac{\sqrt{6n}}{8\pi} m^{-\frac{2}{3}}\right).
\end{align*}

\end{proof}

\section{The Circle Method}\label{sec:Circle}
In this section, as in~\cite{BD13}, we use Wright's variant of the Circle Method to complete the proof of Theorem~\ref{theo:main}.

Using Cauchy's theorem, we write $N(m,n)$ as as integral of its generating function $R_{m}(\tau)$:
$$
N \left( m,n \right) = \frac{1}{2\pi i } \int_C \frac{R_m (\tau)}{q^{n+1}} dq ,
$$
where the contour is the counterclockwise transversal of the circle $C := \{q\in\C\: ;\: |q| = e^{-\beta}\}$. Recall that $s=\beta(1+ixm^{-\frac 13})$. Changing variables we may write
$$
N \left( m,n \right)=\frac{\beta}{2\pi m^\frac 13}\int\limits_{|x|\leq \frac{\pi m^\frac 13}{\beta}} R_m\left(\frac{is}{2\pi}\right)e^{ns} dx.
$$
We split this integral into two pieces
$$
N(m,n) = M + E
$$
with
\begin{align*}
M&:=\frac{\beta}{2\pi m^\frac 13}\int\limits_{|x|\leq 1} R_m\left(\frac{is}{2\pi}\right)e^{ns} dx,\\
E&:=\frac{\beta}{2\pi m^\frac 13}\int\limits_{1\leq |x|\leq \frac{\pi m^\frac 13}{\beta}} R_m\left(\frac{is}{2\pi}\right)e^{ns} dx.
\end{align*}
In the following we show that $M$ contributes to the asymptotic main term whereas $E$ is part of the error term.

As the estimation of $R_m(\tau)$ close to the dominant pole is exactly the same as the one of $\mathcal{C}_{m,1} (q)$ in~\cite{BD13}, the asymptotic behavior of $M$ here is the same as in~\cite{BD13}:

\begin{proposition}\label{integralM}
We have
$$
M =\frac{\beta}{4}\operatorname{sech}^2\left(\frac{\beta m}{2}\right)p(n)\left(1+O\left( \frac{m^\frac 13}{n^\frac 14} \right)\right).
$$
\end{proposition}

Let us now turn to the integral $E$.

\begin{proposition}
\label{integralE}
As $n\rightarrow \infty$
$$
E\ll n^{\frac{1}{2}} \exp \left(\pi \sqrt{\frac{2n}{3}} - \frac{\sqrt{6n}}{8\pi} m^{-\frac{2}{3}}\right).
$$
\end{proposition}

\begin{proof}
Using Proposition~\ref{far}, we may bound
\begin{align*}
E&\ll \frac{\beta}{m^{\frac13}}\int_{1\leq x\leq\frac{\pi m^{\frac13}}{\beta}}  n^{\frac{1}{2}} \exp \left(\pi \sqrt{\frac{n}{6}} - \frac{\sqrt{6n}}{8\pi} m^{\frac{-2}{3}}\right) e^{\beta n}  dx\\
&\ll n^{\frac{1}{2}} \exp \left(\pi \sqrt{\frac{2n}{3}} - \frac{\sqrt{6n}}{8\pi} m^{\frac{-2}{3}}\right).
\end{align*}
\end{proof}
Thus $E$ is exponentially smaller than $M$. This completes the proof of Theorem~\ref{theo:main}.

\providecommand{\bysame}{\leavevmode\hbox to3em{\hrulefill}\thinspace}
\providecommand{\MR}{\relax\ifhmode\unskip\space\fi MR }
\providecommand{\MRhref}[2]{%
  \href{http://www.ams.org/mathscinet-getitem?mr=#1}{#2}
}
\providecommand{\href}[2]{#2}

\end{document}